\documentclass[11pt]{article}
\usepackage{amssymb}
 \usepackage{amsmath,amsthm}
\usepackage{amsmath}
\usepackage{amsthm}
\usepackage{epstopdf}
\usepackage{amsthm}
\usepackage{authblk}
\usepackage[pdfborder={0 0 0}]{hyperref}
\newtheorem{thm}{Theorem}[section]
\newtheorem{remark}{Remark}[section]
\newtheorem{lemma}{Lemma}[section]
\newtheorem{definition}{Definition}[section]

\newtheorem{proposition}{Proposition}[section]
\def \R{\mathbb R} 
\def\RR{{\mathbb R}}
\def\OL{\overline}
\def \UL {\underline}
\def\({\left(}
\def\){\right)}
\def\[{\left[}
\def\]{\right]}
\def \NN{{\mathbb N}}

\newcommand {\WS}{\stackrel {*}{\longrightarrow}}

\begin{document}
\title{Asymptotics of discrete Riesz $d$-polarization on subsets of $d$-dimensional manifolds} %


\author{
S. V. Borodachov\\
Department of Mathematics,
Towson University, \\ Towson, MD, 21252, USA\\
e-mail: {sborodachov@towson.edu}
\and
N. Bosuwan*\\
Department of Mathematics,
Vanderbilt University, \\Nashville, TN, 37240, USA\\
e-mail: {nattapong.bosuwan@vanderbilt.edu}
}

\maketitle

\section*{Abstract}
\noindent We prove a conjecture of T. Erd\'{e}lyi and E.B. Saff, concerning the form of the dominant term (as $N\to \infty$) of the $N$-point Riesz $d$-polarization constant for an infinite compact subset $A$ of a $d$-dimensional $C^{1}$-manifold  embedded in $\mathbb{R}^{m}$ ($d\leq m$).
Moreover, if we assume further that the $d$-dimensional Hausdorff measure of $A$ is positive, we show that any asymptotically optimal sequence of $N$-point configurations for the $N$-point $d$-polarization problem on $A$  is asymptotically uniformly distributed with respect to $\mathcal H_d|_A$.

These results also hold for finite unions of such sets $A$ provided that their pairwise intersections have $\mathcal H_d$-measure zero.

\renewcommand{\thefootnote}{\fnsymbol{footnote}} 
\footnotetext{\emph{Key words and phrases.} Riesz polarization, Chebyshev constants, Riesz energy, potentials} 
\footnotetext{\emph{2000 Mathematics Subject Classifications.} Primary: 31C20, 31C45, Secondary: 28A78 }      
\footnotetext{*This work was included in a part of this author's Ph.D. dissertation under the direction of E.B. Saff at Vanderbilt University with the other author's consent.} 
\renewcommand{\thefootnote}{\arabic{footnote}} 
\section{Introduction}\label {h}

 \quad Let $\omega_N=\{x_1,\ldots, x_N\}$ denote a configuration of $N$ (not necessarily distinct) points in the $m$-dimensional Euclidean space $\mathbb{R}^m$ (such configurations are known as multisets, however, we will still use the word \emph{configurations}). For an infinite compact set $A\subset \mathbb{R}^m$ and $s>0,$ we define the following quantities:
\begin{equation*}
M^{s}(\omega_N;A):=\min_{y\in A} \sum_{i=1}^N\frac{1}{|y-x_i|^{s}}
\end{equation*}
and
\begin{equation}\label {e1}
M^s_N(A):=\max_{\substack{\omega_N \subset A \\ \#\omega_N=N}} M^{s}(\omega_N;A),
\end{equation}
where $\#\omega_N$ stands for the cardinality of the multiset $\omega_N$.
Following \cite{Saff},  we will call $M^{s}_N(A)$ the \emph{$N$-point Riesz $s$-polarization constant of $A.$} The quantity $M^{s}_N(A)$ is also known as
 the \emph{$N^{\textup{th}}$ $L_s$ Chebyshev constant} of the set $A$ (cf. e.g. \cite{AmbrusBallErdelyi}).
We will call an $N$-point configuration $\omega_N\subset A$ {\it optimal for $M^s_N(A)$} if it attains the maximum on the right-hand side of (\ref {e1}).

It is not difficult to verify that for a fixed vector $\textbf{x}_N:=(x_1,\ldots,x_N)$ in $A^{N}$ (the $N$-th Cartesian power of $A$), the potential function
$f(y):=\sum_{i=1}^N |y-x_i|^{-s}, s>0,$
is lower semi-continuous in $y$  on the set $A$ and the function
$g(\textbf{x}_N):=M^{s}(\textbf{x}_N;A), s>0,$
is upper semi-continuous in $\textbf{x}_N$ on $A^N.$ So, the function $f(y)$ attains its minimum on $A$ and the function $g(\textbf{x}_N)$ attains its maximum on $A^N$; i.e. an optimal configuration in (\ref {e1}) exists when $A$ is an infinite compact set.

The $N$-point Riesz $s$-polarization constant was earlier considered by M.~Ohtsuka in \cite{Ohtsuka}. In particular, he showed that for any infinite compact set $A\subset \mathbb{R}^{m}$, the following limit, called the \emph{Chebyshev constant} of $A$, exists as an extended real number:
\begin {equation}\label {e2}
\mathcal{M}^{s}(A):=\lim_{N \rightarrow \infty} \frac{M^{s}_N(A)}{N}.
\end {equation}
Moreover, he showed that $\mathcal{M}^{s}(A)\geq W^s(A)$, where $W^s(A)$ is the Wiener constant of $A$ corresponding to the same value of $s$. Later,  Chebyshev constants were studied in  \cite{FarkasNagy} and \cite{Revesz1} and used to study the so-called \emph{rendezvous} or \emph{average numbers} in  \cite{Revesz2} and \cite{Revesz1}. In particular, it follows from \cite [Theorem 11]{FarkasNagy} that $\mathcal M^s(A)=W^s(A)$ whenever the maximum principle is satisfied on $A$ for the Riesz $s$-potential. More information on the properties of the Wiener constant can be found, for example, in the book \cite {Lan1972}.

The optimality of $N$ distinct equally spaced points on the circle for the Riesz $s$-polarization problem was proved by G. Ambrus in \cite{Ambrus} and by G. Ambrus, K.~Ball, and T. Erd{\'e}lyi in \cite{AmbrusBallErdelyi} for $s=2$. T. Erd{\'e}lyi and E.B. Saff \cite{Saff} established this for $s=4$. For arbitrary $s>0$, this result was proved by D.P. Hardin, A.P. Kendall, and E.B. Saff \cite{HardinKendallSaff} (paper \cite{Nikolov} earlier established this result for $N=3$). Some problems closely related to polarization were considered in \cite{Nikolov1}.

Let $\mathcal{H}_d$ be the $d$-dimensional Hausdorff measure in $\mathbb{R}^{m}$ normalized so that the copy of the $d$-dimensional unit cube embedded in $\R^m$ has measure~$1$.
The inequality $\mathcal{M}^{s}(A)\geq W^s(A)$ implies that on any infinite compact set $A$ of zero $s$-capacity (i.e., when $W^s(A)=\infty$), the limit $\mathcal{M}^s(A)$ is infinite. 
This means that the $N$-point Riesz $s$-polarization constant $M_N^{s}(A)$ grows at a rate faster than $N.$ In particular, it was proved by
T. Erd\'{e}lyi and E.B. Saff \cite[Theorem 2.4]{Saff} that for a compact set $A$ in $\mathbb{R}^m$ of positive $d$-dimensional Hausdorff measure, one has $M^{d}_N(A)=O(N \ln N)$, $N\to\infty$, and  $M^s_N(A)=O(N^{s/d})$, $N\to\infty$, for every $s>d$. The order estimate for $s=d$ is sharp when $A$ is contained in a $d$-dimensional $C^1$-manifold and the order estimate for $s>d$ is sharp when $A$ is $d$-rectifiable (see \cite[Theorem 2.3]{Saff}). We remark that the case $d=1$ of these order estimates when $A$ is a circle was obtained in \cite {AmbrusBallErdelyi}.

Furthermore, when $A$ is the unit ball $\mathbb{B}^{d}$ in $\mathbb{R}^{d}$ or the unit sphere $\mathbb{S}^{d}$ in $\mathbb{R}^{d+1}$, paper \cite {Saff} proves that
 \begin{equation}\label{eq2}
 \lim_{N \rightarrow \infty} \frac{M^{d}_N(\mathbb{B}^{d})}{N \ln N}=1, \quad d\geq 1,
 \end{equation} 
 and 
\begin{equation}\label{eq1}
 \lim_{N \rightarrow \infty} \frac{M^{d}_N(\mathbb{S}^d)}{N \ln N}=\frac {\beta_d}{\mathcal H_d(\mathbb{S}^d)}, \quad d\geq 2,
 \end{equation}
where $\beta_d$ denotes the volume of the $d$-dimensional unit ball $\mathbb {B}^d$. 

When $A$ is an infinite compact subset of a $d$-dimensional $C^1$-manifold, T. Erd\'{e}lyi and E.B. Saff \cite {Saff} also show that
\begin {equation}\label {lower}
\liminf\limits_{N\to\infty}{\frac {M^d_N(A)}{N\ln N}}\geq \frac {\beta_d}{\mathcal H_d(A)}
\end {equation}
and conjecture that the limit of the sequence on the left-hand side of \eqref {lower} exists and equals the right-hand side.

Another interesting fact established in \cite{Saff} is that $M^s_N(\mathbb B^d)=N$ for every $N\geq 1$ and $0<s\leq d-2$ (the maximum principle does not hold for the Riesz $s$-potential in the case $0<s<d-2$). 

A more detailed review of results on polarization can be found, for example, in the papers \cite {AmbrusBallErdelyi}, \cite {Saff}, \cite {FarkasNagy}, and \cite {Revesz1}.

The polarization problem is related to the discrete minimal Riesz energy problem described below.
For a set $X_N=\{x_1,\ldots,x_N\}$ of $N\geq 2$ pairwise distinct points in $\mathbb{R}^{m},$ we define its Riesz $s$-energy by 
$$
E_s(X_N):=\sum_{1 \leq j \not=k \leq N}\frac{1}{|x_j-x_k|^s},
$$
and the \emph {minimum $N$-point Riesz $s$-energy} of a compact  set $A\subset \mathbb{R}^{m}$ is defined as
$$
\mathcal{E}_s(A,N):=\min_{\substack{X_N \subset A \\ \# X_N=N}}E_s(X_N).
$$


D.P. Hardin and E.B. Saff proved in \cite{HardinSaff2005} (see also \cite {HarSaf2004}) that if $A$ is an infinite compact subset of a $d$-dimensional $C^{1}$-manifold embedded in $\mathbb{R}^{m}$ (see Definition \ref {D1}), then\footnote{
The results and techniques of \cite {HardinSaff2005}, in fact, yield relations \eqref {eq3} and \eqref {eqq} under a more general assumption that $A$ is a compact set in $\RR^m$ which for every $\epsilon>0$ can be partitioned into finitely many subsets bi-Lipschitz homeomorphic to some sets from $\RR^d$ with constant $1+\epsilon$ and having boundaries relative to $A$ of $\mathcal H_d$-measure zero (see \cite {BHS}). }   
\begin{equation}\label{eq3}
\lim_{N \rightarrow \infty} \frac{\mathcal{E}_d(A,N)}{N^{2} \ln N}=\frac{\beta_d}{\mathcal{H}_d(A)}.
\end{equation}
Furthermore, if $A$ is as in above condition and $\mathcal H_d(A)>0$, then for any sequence $X_N=\{x_{k,N}\}_{k=1}^N$, $N\in \NN$, of asymptotically $d$-energy minimizing $N$-point configurations in $A$ in the sense that 
 $$
 \lim_{ N \rightarrow \infty} \frac{E_d(X_N)}{\mathcal E_d(A,N)}=1,
 $$ 
 we have
 \begin {equation}\label {eqq}
 \frac{1}{N}\sum_{i=1}^N \delta_{x_{i,N}} \WS \frac{\mathcal{H}_d(\cdot)|_{A}}{\mathcal{H}_d(A)}, \ \ \ N \rightarrow \infty,
 \end {equation}
 in the weak$^\ast$ topology of measures (see Section \ref {n} for the definition). Here $\delta_x$ denotes the unit point mass at the point $x.$




The dominant term of the minimum $s$-energy on $d$-rectifiable closed sets in $\RR^m$ ($s>d$) as well as relation \eqref {eqq} for asymptotically optimal sequences of $N$-point configurations were obtained in \cite {HardinSaff2005} and \cite {BorHarSaf2008}. (In the case $d=1$ these results were earlier established for curves in \cite {MMRS}).

Relations (\ref{eq3}) and (\ref{eqq}) have recently been extended by D.P. Hardin, E.B.~Saff, and J.T. Whitehouse to the case of $A$ being a finite union of compact subsets of $\mathbb{R}^m$ where each compact set is contained in some $d$-dimensional $C^1$-manifold in $\mathbb{R}^m$($d \leq m$) and the pairwise intersections of such compact sets have $\mathcal{H}_d$-measure zero. These authors observed that the methods of \cite {MMRS} could be applied (see \cite{BHS}). 


A detailed review of known results on discrete minimum energy problems can be found, for example, in the book \cite {BHS}.

\section {Notation and definitions}\label {n}

In this section we will mention the main definitions used in the paper.
For a subset $K\subset A$, we will denote by $\partial_ A K$ the boundary of $K$ relative to $A$. 

We say that a sequence $\{\mu_n\}_{n=1}^{\infty}$ of Borel probability measures in $\RR^m$ converges to a Borel probability measure $\mu$ in the weak$^\ast$ topology of measures (and write $\mu_n \WS \mu$, $n\to\infty$) if for every continuous function $f:\RR^m\to \RR$,
\begin {equation}\label {eqeq}
\int{f\ \! d\mu_n}\to \int f\ \! d\mu,\ \ \ n\to\infty.
\end {equation}

\begin {remark}\label {R2.1}
{\rm 
It is well known that to prove \eqref {eqeq} when $\mu$ and all the measures $\mu_n$ are supported on a compact set $A\subset \RR^m$,  it is sufficient to show that
$$
\mu_n(K)\to \mu (K),\ \ \ n\to \infty,
$$
for every closed subset $K$ of $A$ with $\mu(\partial _A K)=0$.
}
\end {remark}

It will be convenient to use throughout this paper the following definition of a $d$-dimensional $C^{1}$-manifold in $\mathbb{R}^m$ (see, for example, \cite [Chapter 5]{Spivak}).


 \begin{definition}\label{D1}
{\rm
A set $W \subset \mathbb{R}^m$ is called a \emph{$d$-dimensional $C^1$-manifold embedded in $\mathbb{R}^m$}, $d\leq m$, if every point $y\in W$ has an open neighborhood $V$ relative to $W$ such that $V$ is homeomorphic to an open set $U\subset \mathbb{R}^d$ with the homeomorphism $f:U\to V$ being a $C^1$-continuous mapping and the Jacobian matrix 
$$
J^{f}_{x}:=\begin{bmatrix} ~\nabla f_1(x) ~ \\ \ldots \\ \nabla f_m(x)\end{bmatrix}
$$
of the function $f$ having rank $d$ at any point $x\in U$ (here $f_1,\ldots,f_m$ denote the coordinate mappings of $f$).
}
\end{definition} 

Finally, we call a sequence $\{\omega_N\}_{N=1}^{\infty}$ of $N$-point configurations on $A$ {\it asymptotically optimal for the $N$-point $d$-polarization problem on $A$} if
$$
\lim\limits_{N\to\infty}{\frac {M^d(\omega_N;A)}{M^d_N(A)}}=1.
$$


 \section{Main results}

In this paper we extend relation (\ref{eq2}) to the case of an arbitrary infinite compact set in $\R^d$ and relation (\ref{eq1}) to the case when $A$ is an infinite compact subset of a $d$-dimensional $C^{1}$-manifold embedded in $\mathbb{R}^m$ where $m>d$ or a finite union of such sets provided that their pairwise intersections have $\mathcal H_d$-measure zero. Under additional assumption that $\mathcal{H}_{d}(A)>0,$ we also determine the weak$^\ast$ limiting distribution of asymptotically optimal $N$-point configurations for the $N$-point $d$-polarization problem on these classes of sets. Relation \eqref {add_17} below proves the conjecture made by T. Erd\'{e}lyi and E.B.~Saff in 
 \cite[Conjecture 2]{Saff}.


 \begin{thm}\label{bavv}
 Let $A=\cup_{i=1}^l A_i$ be an infinite subset of $\mathbb{R}^m$, where each set $A_i$ is a compact subset contained in some $d$-dimensional $C^1$-manifold in $\mathbb{R}^m,$ $d\leq m,$ and $\mathcal H_d(A_i\cap A_j)=0$, $1\leq i<j\leq l$. 
 Then
 \begin {equation}\label {add_17}
 \lim_{N \rightarrow \infty} \frac{M^{d}_N(A)}{N \ln N} =\frac{\beta_d}{\mathcal{H}_d{(A)}}.
 \end {equation}
 Furthermore, under an additional assumption that $\mathcal{H}_{d}(A)>0,$  if  $\omega_N=\{x_{i,N}\}_{i=1}^{N}$, $N\in \NN$, is a sequence of asymptotically optimal configurations for the $N$-point $d$-polarization problem on $A$, then in the weak$^\ast$ topology of measures we have \begin{equation}\label{eq7}
\frac{1}{N} \sum_{i=1}^N\delta_{x_{i,N}} \WS \frac{\mathcal{H}_{d}(\cdot)|_{A}}{\mathcal{H}_{d}(A)}, \ \ \ N \rightarrow \infty.
 \end{equation}
\end {thm}
\begin{remark}\label{ujkiol}
\textup{Note that the conditions imposed on the set $A$ imply $\mathcal{H}_d(A)<\infty$. Moreover, if $\mathcal{H}_d(A)=0,$ then the limit in (\ref{add_17}) is understood to be $\infty$.}
\end{remark}

To establish Theorem \ref {bavv} we will use the result proved in Section \ref {4}, Lemma \ref {add_7}, and Proposition \ref {add_log_energy}.

\section {Upper estimate}\label {4}

For a compact set $A\subset \RR^m$, define the quantity
\begin{equation}\label{eq5}
\overline{\alpha}_d(A;\varepsilon):=\sup_{0<r\leq \varepsilon} \sup_{x\in A} \frac{\mathcal{H}_d(B(x,r) \cap A)}{\beta_d r^d}.
\end{equation} 
Let also
$$
\UL h_d(A):=\liminf\limits_{N\to\infty}{\frac {M^d_N(A)}{N\ln N}} \ \ \ {\rm and}\ \ \ \OL h_d(A):=\limsup\limits_{N\to\infty}{\frac {M^d_N(A)}{N\ln N}}.
$$
The main result of this section is given below.
\begin {thm}\label {upper}
Let $d,m\in \NN$, $d\leq m$, and $A\subset \RR^m$ be a compact set with $0<\mathcal H_d(A)<\infty$, containing a closed subset $B$ of zero $\mathcal H_d$-measure such that every compact subset $K\subset A\setminus B$ satisfies 
\begin {equation}\label {3q}
\lim\limits_{\epsilon\to 0^+}\OL \alpha_d(K;\epsilon)\leq 1.
\end {equation}
Then
\begin {equation}\label {add_eq3}
\OL h_d(A)\leq \frac {\beta_d}{\mathcal H_d(A)}.
\end {equation}
If an equality holds in \eqref {add_eq3}, then any infinite  sequence   $\omega_N=\{x_{k,N}\}_{k=1}^N$, $N\in \mathcal N\subset \NN$, of configurations on $A$ such that 
\begin{equation}\label{eq6}
\lim_{N \rightarrow \infty\atop N\in \mathcal N} \frac{M^{d}(\omega_N;A)}{N \ln N} =\frac{\beta_d}{\mathcal{H}_d{(A)}}
\end{equation}
satisfies
\begin {equation}\label {11a}
\frac{1}{N}\sum_{i=1}^N \delta_{x_{i,N}} \WS \frac{\mathcal{H}_d(\cdot)|_{A}}{\mathcal{H}_d(A)}, \ \ \ \mathcal N\ni N \rightarrow \infty.
\end {equation}
\end {thm}
\noindent We precede the proof of Theorem \ref {upper} with the following auxiliary statements.
\begin {lemma}\label {add_up}
Let $0<R\leq r$, $D\subset \RR^m$ be a compact set with $\mathcal H_d(D)<\infty$, $d\in \NN$, $d\leq m$, and $y\in D$. Then
$$
\int\limits_{D\setminus B(y,R)}{\frac {d\mathcal H_d(x)}{\left|x-y\right|^d}}\leq r^{-d} \mathcal{H}_d(D)+\beta_d \overline{\alpha}_d(D;r)  \ln \(\frac {r}{R}\)^d.
$$
\end {lemma}
\begin {proof}
We have
\begin{align}\label{eq11}
\int\limits_{D\setminus B(y,R)}{\frac {d\mathcal H_d(x)}{\left|x-y\right|^d}} &= \int_{0}^{\infty} \mathcal{H}_d\{x \in D\setminus B(y,R) : |x-y|^{-d}>t\} dt\notag\\
&= \int_{0}^{\infty} \mathcal{H}_d\{x\in D\setminus B(y,R) : {t^{-1/d}} > |x-y|\} dt \notag\\
&\leq  \int_{0}^{R^{-d}} \mathcal{H}_d (B(y, t^{-1/d})\cap D) dt \notag\\
&\leq  r^{-d} \mathcal{H}_d(D)+ \int_{r^{-d}}^{R^{-d}} \mathcal{H}_d ( B(y,t^{-1/d}) \cap D) dt \notag\\
&\leq  r^{-d} \mathcal{H}_d(D)+ \beta_d\int_{r^{-d}}^{R^{-d}} \overline{\alpha}_d(D;r)  t^{-1} dt \notag\\
&=r^{-d} \mathcal{H}_d(D)+\beta_d \overline{\alpha}_d(D;r)  \ln \(\frac {r}{R}\)^d,\nonumber
\end{align}
which completes the proof.
\end {proof}

\begin {lemma}\label {add_10}
Let $d,m\in \NN$, $d\leq m$, and $A\subset \RR^m$ be a compact set with $0<\mathcal H_d(A)<\infty$, containing a closed subset $B$ of zero $\mathcal H_d$-measure such that every compact subset of the set $A\setminus B$ satisfies \eqref {3q}.
Then for any infinite sequence
$\{\omega_N\}_{N\in \mathcal N}$, $\mathcal N\subset \NN$, of $N$-point configurations on the set $A$, the inequality
\begin {equation}\label {add_e1}
\frac {\mathcal H_d(K)}{\beta_d}\cdot\liminf\limits_{N\to\infty\atop N\in \mathcal N}\frac {M^d(\omega_N;A)}{N\ln N}\leq \liminf\limits_{N\to\infty\atop N\in \mathcal N}\frac {\# (\omega_N \cap K)}{N}
\end {equation}
holds for any compact subset $K\subset A$ with $\mathcal H_d(K)>0$ and $\mathcal H_d(\partial _A K)=0$.
\end {lemma}
\begin {proof}
Without loss of generality, we can assume that $B\neq \emptyset$ since in the case $B=\emptyset$ we can also use as $B$ any non-empty compact subset of $A$ with $\mathcal H_d(B)=0$.

Let $x_{1,N},\ldots,x_{N,N}$ be the points in the configuration $\omega_N$, $N\in \mathcal N$, and let $K\subset A$ be any compact subset of positive $\mathcal H_d$-measure such that $\mathcal H_d(\partial _A K)=0$. Denote 
$$
K_\rho:=\{x\in K : {\rm dist}(x,B\cup \partial _A K)\geq \rho\},\ \ \ \rho>0. 
$$
Choose an arbitrary number $\rho>0$ such that $\mathcal H_d(K_{2\rho})>0$.  
Let $r>0$ be any number such that $2\beta_d r^d<\mathcal H_{d}(K_{2\rho})$. For each $j=1,\ldots,N$, define the set
$$
\mathcal D_{j,N}:=K_{2\rho}\setminus B(x_{j,N},rN^{-1/d}) \ \ \ {\rm and \ let}\ \ \ \mathcal D_N:=\bigcap \limits_{j=1}^{N}{\mathcal D_{j,N}}.
$$
Notice that ${\rm dist}(K_{2\rho},K\setminus K_\rho)\geq \rho>0$. Furthermore, ${\rm dist}(K_{2\rho},A\setminus K)>0$. Indeed, if there were sequences $\{x_n\}$ in $K_{2\rho}$ and $\{y_n\}$ in $A\setminus K$ such that $\left|x_n-y_n\right|\to 0$, $n\to\infty$, then by compactness of $K_{2\rho}$ and $A$ there would exist subsequences $\{x_{n_k}\}$ and $\{y_{n_k}\}$ having the same limit $z\in K_{2\rho}$. Since $\{y_{n_k}\}\subset A\setminus K$  the point $z$ must belong to $\partial _A K$, which contradicts to the definition of the set $K_{2\rho}$.
Thus, we have
$$
h:={\rm dist}(K_{2\rho},A\setminus K_\rho)=\min\{{\rm dist}(K_{2\rho},K\setminus K_\rho),{\rm dist}(K_{2\rho},A\setminus K)\}>0.
$$
Choose $N\in \mathcal N$ to be such that $rN^{-1/d}<h$ and $\OL \alpha_d(K_\rho;rN^{-1/d})\leq 2$ (such $N$ exists since $K_\rho$ is a compact subset of $A\setminus B$, and by assumption, satisfies $\lim _{N\to\infty} \OL \alpha_d(K_\rho;r N^{-1/d})\leq 1$). Then
$$
\mathcal H_d(\mathcal D_N)=\mathcal H_d\(K_{2\rho}\setminus \bigcup_{j=1}^{N}{B(x_{j,N},rN^{-1/d})}\)
$$
$$
=\mathcal H_d\(K_{2\rho}\setminus \bigcup_{x_{j,N}\in K_\rho}{B(x_{j,N},rN^{-1/d})}\)
$$
$$
\geq \mathcal H_d(K_{2\rho})-\sum\limits_{x_{j,N}\in K_\rho}{\mathcal H_d\(K_\rho\cap B(x_{j,N},rN^{-1/d})\)}
$$
$$
\geq \mathcal H_d(K_{2\rho})-\beta_d r^d \frac {\# (\omega_N \cap K_\rho)}{N}\cdot \OL \alpha_d(K_\rho;rN^{-1/d})
$$
$$
\geq \mathcal H_d(K_{2\rho})-\beta_d r^d \OL \alpha_d(K_\rho;rN^{-1/d})\geq \mathcal H_d(K_{2\rho})-2\beta_d r^d=:\gamma _{r,\rho}>0.
$$
Let $\widetilde {\mathcal D}_{j,N}:=K_\rho\setminus B(x_{j,N},rN^{-1/d})$. Then
$$
M^d(\omega_N;A)=\min\limits_{x\in A}\sum\limits_{j=1}^{N}{\frac {1}{\left|x-x_{j,N}\right|^d}}
$$
$$
\leq \frac {1}{\mathcal H_d(\mathcal D_N)}\sum\limits_{j=1}^{N}{\ \int\limits_{\mathcal D_N}{\frac {d\mathcal H_d(x)}{\left|x-x_{j,N}\right|^d}}}
\leq\frac {1}{\gamma_{r,\rho}}\sum\limits_{j=1}^{N}\ \int\limits_{\mathcal D_{j,N}}{{\!\!\frac {d\mathcal H_d(x)}{\left|x-x_{j,N}\right|^d}}}
$$
$$
\leq\frac {1}{\gamma_{r,\rho}}\(\sum\limits_{x_{j,N}\in K_\rho}\ \int\limits_{\widetilde {\mathcal D}_{j,N}}{\!\!\frac {d\mathcal H_d(x)}{\left|x-x_{j,N}\right|^d} }+\!\!\sum\limits_{x_{j,N}\in A\setminus K_\rho}{\ \int\limits_{\mathcal D_{j,N}}{\!\!\frac {d\mathcal H_d(x)}{\left|x-x_{j,N}\right|^d}}}\).
$$
Taking into account Lemma \ref {add_up} with $R=rN^{-1/d}$ and $D=K_\rho$ and the fact that ${\rm dist}(\mathcal D_{j,N},A\setminus K_\rho)\geq {\rm dist}(K_{2\rho},A\setminus K_\rho)= h>0$, we will have
$$
M^d(\omega_N;A)\leq \frac {1}{\gamma_{r,\rho}} \Bigg( \# (\omega_N\cap K_\rho)\(\frac {\mathcal H_d(K_\rho)}{r^d}+\beta_d\OL \alpha_d(K_\rho;r)\ln N\)
$$
$$
+\sum\limits_{x_{j,N}\in A\setminus K_\rho}\frac {\mathcal H_d(\mathcal D_{j,N})}{h^d} \Bigg).
$$
Consequently,
\begin {equation}\label {add_eq2}
\frac {M^d(\omega_N;A)}{N\ln N}\leq \frac {1}{\gamma_{r,\rho}}\(\frac {\# (\omega_N\cap K_\rho)}{N}\(\frac {\mathcal H_d(K_\rho)}{r^d\ln N}+\beta_d\OL \alpha_d(K_\rho;r)\)+\frac {\mathcal H_d(A)}{h^d\ln N} \).
\end {equation}
Passing to the lower limit in \eqref {add_eq2} we will have
$$
\tau:=\liminf\limits_{N\to\infty\atop N\in \mathcal N}{\frac {M^d(\omega_N;A)}{N\ln N}}
\leq \frac {\beta_d \OL \alpha_d(K_\rho;r)}{\mathcal H_d(K_{2\rho})-2\beta_d r^d}\liminf\limits_{N\to\infty\atop N\in \mathcal N}{\frac {\# (\omega_N\cap K_\rho)}{N}}.
$$
Letting $r\to 0$ and taking into account \eqref {3q} and the fact that $K_\rho \subset K$, we will have
$$
\tau\leq \frac {\beta_d}{\mathcal H_d(K_{2\rho})}\liminf\limits_{N\to\infty\atop N\in \mathcal N}{\frac {\# (\omega_N\cap K_\rho)}{N}}\leq \frac {\beta_d}{\mathcal H_d(K_{2\rho})}\liminf\limits_{N\to\infty\atop N\in \mathcal N}{\frac {\# (\omega_N\cap K)}{N}}.
$$
Since $\lim\limits_{\rho \to 0^+}\mathcal H_{d}(K_{2\rho})=\mathcal H_d(K\setminus (B\cup \partial_ A K))=\mathcal H_d(K)$, we finally have
$$
\tau\leq \frac {\beta_d}{\mathcal H_d(K)}\liminf\limits_{N\to\infty\atop N\in \mathcal N}{\frac {\# (\omega_N\cap K)}{N}},
$$
which implies \eqref {add_e1}.
\end {proof}

\noindent {\bf Proof of Theorem \ref {upper}.} Let $\mathcal N_0\subset \NN$ be an infinite subset such that 
$$
\OL h_d(A)=\lim\limits_{N\to\infty\atop N\in \mathcal N_0}\frac {M^d_N(A)}{N\ln N}.
$$
Let $\{\OL\omega_N\}_{N\in \mathcal N_0}$ be a sequence of $N$-point configurations on $A$ such that $M^d_N(A)=M^d(\OL\omega_N;A)$, $N\in \mathcal N_0$. Then applying Lemma \ref {add_10} with $K=A$, we will have
$$
\OL h_d(A)=\lim\limits_{N\to\infty\atop N\in \mathcal N_0}\frac {M^d(\OL\omega_N;A)}{N\ln N}\leq \frac {\beta_d}{\mathcal H_d(A)}\liminf\limits_{N\to\infty\atop N\in \mathcal N_0}{\frac {\# (\overline{\omega}_N\cap A)}{N}}=\frac {\beta_d}{\mathcal H_d(A)}
$$
and inequality \eqref {add_eq3} follows.

Assume now that $\OL h_d(A)=\beta_d \mathcal H_d(A)^{-1}$ and let $\{\omega_N\}_{N\in \mathcal N}$, $\mathcal N\subset \NN$, be any infinite sequence of $N$-point configurations on $A$ satisfying \eqref {eq6}.  
For any closed subset $D\subset A$ with $\mathcal H_d(D)>0$ and $\mathcal H_d(\partial _A D)=0$, by Lemma \ref {add_10} we have
\begin {equation}\label {add_q}
\liminf\limits_{N\to\infty\atop N\in \mathcal N}\frac {\# (\omega_N \cap D)}{N}\geq \frac {\mathcal H_d(D)}{\beta_d}\lim\limits_{N\to\infty\atop N\in \mathcal N}{\frac {M^d(\omega_N;A)}{N\ln N}}= \frac {\mathcal H_d(D)}{\mathcal H_d(A)}.
\end {equation}

Let now $P\subset A$ be any closed subset of zero $\mathcal H_d$-measure. Show that
\begin {equation}\label {add_p}
\lim\limits_{N\to \infty\atop N\in \mathcal N}{\frac {\# (\omega_N \cap P)}{N}}=0.
\end {equation}
If $P=\emptyset$, then \eqref {add_p} holds trivially.
Let $P\neq \emptyset$. Since $\mathcal H_d(A)<\infty$, for every $\epsilon>0$, there are at most finitely many numbers $\delta>0$ such that the set $P[\delta]:=\{x\in A : {\rm dist}(x,P)=\delta\}$ has $\mathcal H_d$-measure at least $\epsilon$. This implies that there are at most countably many numbers $\delta>0$ such that $\mathcal H_d(P[\delta])>0$. Denote also $P_\delta=\{x\in A : {\rm dist}(x,P)\geq \delta\}$, $\delta>0$. Then there exists a positive sequence $\{\delta_n\}$ monotonically decreasing to $0$ such that every set $\partial _A P_{\delta_n}\subset P[\delta_n]$ has $\mathcal H_d$-measure zero. Since $P_{\delta_n}$ is closed and $\mathcal H_d(P_{\delta_n})>0$ for every $n$ greater than some $n_1$, in view of \eqref {add_q}, we have
$$
\liminf\limits_{N\to\infty\atop N\in \mathcal N}\frac {\# (\omega_N \cap (A\setminus P))}{N}\geq \liminf\limits_{N\to\infty\atop N\in \mathcal N}\frac {\# (\omega_N \cap P_{\delta_n})}{N}\geq \frac {\mathcal H_d(P_{\delta_n})}{\mathcal H_d(A)},\ \ \ n>n_1.
$$
Since $\mathcal H_d(P_{\delta_n})\to \mathcal H_{d}(A\setminus P)=\mathcal H_d(A)$, $n\to\infty$, we have
$$
\lim\limits_{N\to \infty\atop N\in \mathcal N}{\frac {\# (\omega_N \cap (A\setminus P))}{N}}=1,
$$
which implies \eqref {add_p}.

Since the set $\OL {A\setminus D}$ is also a closed subset of $A$ and $\mathcal H_d(\partial _A (A\setminus D))=\mathcal H_d(\partial _A D)=0$, by \eqref {add_q} and \eqref {add_p} (with $P=\partial_A D$) we have
$$
\limsup\limits_{N\to\infty\atop N\in\mathcal N}{\frac {\# (\omega_N \cap D)}{N}}=1-\liminf\limits_{N\to\infty\atop N\in\mathcal N}{\frac {\# (\omega_N \cap (A\setminus D))}{N}}
$$
$$
= 1-\liminf\limits_{N\to\infty\atop N\in\mathcal N}{\frac {\# (\omega_N \cap \OL {A\setminus D})}{N}}\leq 1-\frac {\mathcal H_d(\OL {A\setminus D})}{\mathcal H_d(A)}=\frac {\mathcal H_d(D)}{\mathcal H_d(A)}.
$$
Thus,
\begin {equation}\label {15}
\lim\limits_{N\to\infty\atop N\in \mathcal N}{\frac {\# (\omega_N \cap D)}{N}}=\frac {\mathcal H_d(D)}{\mathcal H_d(A)}
\end {equation}
for any closed subset $D\subset A$ with $\mathcal H_d(D)>0$ and $\mathcal H_d(\partial _A D)=0$. In view of \eqref {add_p} relation \eqref {15} also holds when $D\subset A$ is closed and $\mathcal H_d(D)=0$. Then in view of Remark \ref {R2.1} we have \eqref {11a}.\hfill $\square$

\section {Auxiliary statements}


We will show in this section that for every set $A$ satisfying the assumptions of Theorem \ref {bavv}, the assumptions of Theorem \ref {upper} necessarily hold.

\begin{proposition}\label{Pr1}
Let $A$ be a compact subset of a $d$-dimensional $C^1$-manifold embedded in $\R^m$, $d\leq m$. Then for such a set $A,$ 
\begin{equation}\label{eq4}
\lim_{\varepsilon \rightarrow 0^{+}} \overline{\alpha}_d(A;\varepsilon) \leq1.
\end{equation}
\end{proposition}
The proof of this statement is given in the Appendix.

\begin {lemma}\label {add_7}
Let $A=\cup_{i=1}^{l}A_i$, where each set $A_i$ is a compact set contained in some $d$-dimensional $C^1$-manifold in $\RR^m$, $d\leq m$, and $\mathcal H_d(A_i\cap A_j)=0$, $1\leq i< j\leq l$. Then there is a compact subset $B\subset A$ with $\mathcal H_d(B)=0$ such that every compact subset $K\subset A\setminus B$ satisfies $\lim\limits_{\epsilon\to 0^+}{\OL \alpha_d(K;\epsilon)}\leq 1$.
\end {lemma}
\begin {proof}
Denote $B:=\bigcup\limits_{1\leq i<j\leq l}{A_i\cap A_j}$. Let $K\subset A\setminus B$ be a compact subset. Then
$$
\delta_0:=\min\limits_{1\leq i<j\leq l}{{\rm dist}(A_i\cap K,A_j\cap K)}>0.
$$
Choose any $\epsilon\in (0,\delta_0)$. Choose also arbitrary $r\in (0,\epsilon]$ and $x\in K$. We have $x\in A_i$ for some $1\leq i\leq l$ and $x\notin A_j$ for every $j\neq i$. Since $r<\delta_0$, we have $B(x,r)\cap K\subset B(x,r)\cap A_i$ and consequently,
$$ 
\frac {\mathcal H_d(B(x,r)\cap K)}{\beta_d r^d}\leq \frac {\mathcal H_d(B(x,r)\cap A_i)}{\beta_d r^d}
$$
$$
\leq \sup\limits_{t\in (0,\epsilon]}{\sup\limits_{y\in A_i}}\frac {\mathcal H_d(B(y,t)\cap A_i)}{\beta_d t^d}=\OL \alpha_d(A_i;\epsilon)\leq \max\limits_{1\leq j\leq l}{\OL \alpha_d(A_j;\epsilon)}.
$$
Consequently,
\begin {equation}\label {add_5}
\OL \alpha_d(K;\epsilon)=\sup\limits_{r\in (0,\epsilon]}{\sup\limits_{x\in K}}\frac {\mathcal H_d(B(x,r)\cap K)}{\beta_d r^d}\leq \max\limits_{1\leq j\leq l}{\OL \alpha_d(A_j;\epsilon)}.
\end {equation}
Since each $A_i$ is a compact subset of a $d$-dimensional $C^1$-manifold, by Proposition \ref {Pr1}, we have $\lim\limits_{\epsilon\to 0^+}\OL \alpha_d(A_i;\epsilon)\leq 1$, $i=1,\ldots,l$. Then in view of \eqref {add_5} we have $\lim\limits_{\epsilon\to 0^+}\OL \alpha_d(K;\epsilon)\leq 1$.
\end {proof}


The following proposition is a part of the result by D.P. Hardin, E.B.~Saff, and J.T. Whitehouse mentioned at the end of Section \ref {h}. For completeness, we will reproduce its proof.
\begin {proposition}\label {add_log_energy}
Let $A=\cup_{i=1}^{l}A_i$, where each $A_i$ is a compact set contained in some $d$-dimensional $C^1$-manifold in $\RR^m$ and $\mathcal H_d(A_i\cap A_j)=0$, $1\leq i< j\leq l$. Then
$$
\UL g_d(A):=\liminf\limits_{N\to\infty}{\frac {\mathcal E_d(A,N)}{N^2\ln N}}\geq \frac {\beta_d}{\mathcal H_d(A)}.
$$
\end {proposition}
\begin {proof} 
Since every set $A_i$ is a compact subset of a $d$-dimensional $C^1$-manifold, in view of Theorem 2.4 in \cite {HardinSaff2005}, there holds
$
\UL g_d(A_i)\geq \beta_d \mathcal H_d(A_i)^{-1},
$ $i=1,\ldots,l$.
In view of inequality (34) from Lemma 3.2 in \cite {HardinSaff2005}, we then have
$$
\UL g_d(A)=\UL g_d\(\bigcup\limits_{i=1}^{l}{A_i}\)\geq \(\sum\limits_{i=1}^{l}\UL g_d(A_i)^{-1}\)^{-1}\!\!\geq \(\frac {1}{\beta_d}\sum\limits_{i=1}^{l}\mathcal H_d(A_i)\)^{-1}\!\!=\frac {\beta_d}{\mathcal H_d(A)},
$$
which yields the desired inequality.
\end {proof}

\section{Proof of Theorem \ref {bavv}}


\begin{proof}
The proof of the lower estimate in \eqref {add_17} will repeat the proof of inequality (2.9) in \cite {Saff}.
It is known that (see \cite {Saff}, \cite{FarkasNagy}, or \cite{Revesz1}) for any infinite compact set $A\subset \mathbb{R}^m,$
\begin{equation}\label{eq15}
M^{s}_N(A) \geq \frac{1}{N-1} \mathcal{E}_s(A,N), \quad N \geq 2,\ \ \ s>0. 
\end{equation}


Then Proposition \ref {add_log_energy} and inequality (\ref{eq15}) give  the lower estimate for $M^{d}_N(A)$:
$$\liminf_{N \rightarrow \infty} \frac{M_N^d(A)}{N \ln N} \geq\liminf_{N \rightarrow \infty} \frac{\mathcal{E}_d(A,N)}{(N-1)N \ln N} \geq\frac{\beta_d}{\mathcal{H}_d(A)}.
$$
Note that if $\mathcal{H}_d(A)=0,$ then $\lim_{N \rightarrow \infty} {M_N^d(A)}/{(N \ln N)}=\infty.$

Now, assume that $\mathcal{H}_d(A)>0.$ In view of Lemma \ref {add_7} and Remark \ref{ujkiol}, the set $A$ satisfies the assumptions of Theorem \ref{upper}. Consequently
$$\limsup_{N \rightarrow \infty} \frac{M_N^d(A)}{N \ln N} \leq \frac{\beta_d}{\mathcal{H}_d(A)}.$$
This implies \eqref {add_17}.

Every sequence $\{\omega_N\}_{N=1}^{\infty}$ of $N$-point configurations, which is asymptotically optimal for the $N$-point $d$-polarization problem on $A$ must satisfy \eqref {eq6} with $\mathcal N=\NN$. Since $\OL h_d(A)=\beta_d\mathcal H_d(A)^{-1}$, by Theorem \ref {upper} we obtain \eqref {eq7}.
\end{proof}


\section{Appendix}

In this part of the paper we prove Proposition \ref {Pr1}.

We say that a set $B$ in $\mathbb{R}^m$ is {\it bi-Lipschitz homeomorphic to a set $D\subset \mathbb{R}^n$ with a constant $M\geq 1$}, if there is a mapping $\varphi:B\to D$ such that $\varphi(B)=D$ and
$$
M^{-1}\left|x-y\right|\leq \left|\varphi(x)-\varphi(y)\right|\leq M\left|x-y\right|,\ \ \ x,y\in B.
$$

\begin{lemma}\label{lemma1}
Let $U\subset \mathbb{R}^d$ be a non-empty open set and $f:U\to \mathbb{R}^m$, $m\geq d$, be 
an injective $C^1$-continuous mapping such that its inverse $f^{-1}:f(U)\to U$ is continuous and the Jacobian matrix 
\begin{equation}\label{tgfr}
J^{f}_{x}:=\begin{bmatrix} ~\nabla f_1(x) ~ \\ \ldots \\ \nabla f_m(x)\end{bmatrix}
\end{equation}
of $f$ has rank $d$ at any point $x\in U$.
Then for every $\epsilon>0$ and every point $y_0\in f(U)$, there is a closed ball $B$ centered at $y_0$ such that the set $B\cap f(U)$ is bi-Lipschitz homeomorphic to some compact set in $\mathbb{R}^d$ with a constant $1+\epsilon$.  
\end{lemma}

\begin{proof}[Proof of Lemma \ref{lemma1}]

Let $x_0\in U$ be the point such that $f(x_0)=y_0$. Choose any $\epsilon>0$ and let $\delta=\delta (x_0,\epsilon)>0$ be such that $B[x_0,\delta]\subset U$ and 
$$
\left|\nabla f_i(x)-\nabla f_i(x_0)\right|<\epsilon,\ \ \ x\in B[x_0,\delta],\  i=1,\ldots,m.
$$
Let $x,y\in B[x_0,\delta]$ be two arbitrary points. Define the function $g_i(t):=f_i(x+t(y-x))$, $t\in [0,1]$. Then there exists $\xi_i\in (0,1)$ such that
$$
f_i(y)-f_i(x)=g_i(1)-g_i(0)=g'_i(\xi_i)=\nabla f_i(z_i)\cdot (y-x)
$$
$$
=\nabla f_i(x_0)\cdot (y-x)+\(\nabla f_i(z_i)-\nabla f_i(x_0)\)\cdot (y-x),
$$
where $z_i=x+\xi_i(y-x)$, $i=1,\ldots,m$. Since $z_i\in B[x_0,\delta]$, we have
$$
\left|f_i(y)-f_i(x)-\nabla f_i(x_0)\cdot (y-x)\right|
$$
$$
=\left|\(\nabla f_i(z_i)-\nabla f_i(x_0)\)\cdot (y-x)\right|\leq \epsilon \left|y-x\right|,\ \ \ \ i=1,\ldots,m,
$$ 
and hence (we treat $x$ and $y$ as vector-columns below),
\begin{equation}\label{eq8}
\left|f(y)-f(x)-J^{f}_{x_0}(y-x)\right|\leq \epsilon\sqrt m  \left|y-x\right|, \ \ x,y\in B[x_0,\delta].
\end{equation}
Since the matrix $J^{f}_{x_0}$ has rank $d$, for every standard basis vector $e_i$ from $\R^d$, there is a vector $v_i\in R^m$ such that $(J^{f}_{x_0})^T v_i=e_i$, $i=1,\ldots,d$, where $(J^{f}_{x_0})^T$ denotes the transpose of the matrix $J^{f}_{x_0}$. Then the $d\times m$ matrix $Z:=\[v_1,\ldots,v_d\]^T$ satisfies $ZJ^{f}_{x_0}=I_d$, where $I_d$ is the $d\times d$ identity matrix. Taking into account \eqref {eq8} we have
$$
\left|f(y)-f(x)-J^{f}_{x_0}(y-x)\right|\leq \epsilon\sqrt m  \left|ZJ^{f}_{x_0}(y-x)\right|
$$
$$
\leq \epsilon\sqrt m  \|Z\| \left|J^{f}_{x_0}(y-x)\right|,\ \ \ x,y\in B[x_0,\delta],
$$
where 
$
\|Z\|:=\max \{\left|Zu\right| : u\in \R^m,\ \left|u\right|=1\}. 
$
Consequently,
$$
\(1-\epsilon \sqrt{m} \|Z\|\)\left|J^{f}_{x_0}(y-x)\right|\leq \left|f(y)-f(x)\right|
$$
$$
\leq \(1+\epsilon \sqrt{m} \|Z\|\)\left|J^{f}_{x_0}(y-x)\right|,\ \ \ 
x,y\in B[x_0,\delta].
$$

Let $u_1,\ldots,u_d$ be an orthonormal basis in the subspace $H$ of $R^m$ spanned by the columns of the matrix $J^{f}_{x_0}$ and let $D:=\[u_1,\ldots,u_d\]$ be the $m\times d$ matrix with columns $u_1,\ldots, u_d$. Since the columns of $J^{f}_{x_0}$ also form a basis in $H$, there exists an invertible $d\times d$ matrix $Q$ such that $D=J^{f}_{x_0}Q$. 

Let $V\subset \R^d$ be the open set such that $\Phi(V)=B(x_0,\delta)$, where $\Phi:\R^d\to \R^d$ is the linear mapping given by $\Phi(v)=Qv$. Since the columns of the matrix $D$ are orthonormal, for every $u,v\in \OL V$, we will have
$$
\left|f\circ \Phi(u)-f\circ \Phi(v)\right|=\left|f(Qu)-f(Qv)\right|
$$
$$
\leq \(1+\epsilon \sqrt{m} \|Z\|\)\left|J^{f}_{x_0}Q(u-v)\right|=\(1+\epsilon \sqrt{m} \|Z\|\)\left|D(u-v)\right|
$$
$$
= \(1+\epsilon \sqrt{m} \|Z\|\)\left|u-v\right|.
$$
Similarly,
$$
\left|f\circ \Phi(u)-f\circ \Phi(v)\right|\geq \(1-\epsilon \sqrt{m} \|Z\|\)\left|u-v\right|,\ \ \ u,v\in \OL V,
$$
which implies that for $0<\epsilon <(\sqrt {m}\|Z\|)^{-1}$, the restriction of the mapping $\psi:=f\circ \Phi$ to the set $\OL V$ is a bi-Lipschitz mapping onto the set $f(\Phi(\OL V))=f(B[x_0,\delta])$ with constant $M_\epsilon:=\max\{ 1+\epsilon \sqrt{m} \|Z\|,(1-\epsilon \sqrt {m}\|Z\|)^{-1}$\}.

Since $f$ is a homeomorphism of $U$ onto $f(U)$, the set $f(B(x_0,\delta))$ is open relative to $f(U)$. Then there is a closed ball $B$ in $\R^m$ centered at $y_0=f(x_0)$ such that $B\cap f(U)\subset f(B(x_0,\delta))$. Then the set $B\cap f(U)=B\cap f(B[x_0,\delta])$ is bi-Lipschitz homeomorphic (with constant $M_\epsilon$) to the set 
$$
V_1:=\psi^{-1}(B\cap f(U))=\psi^{-1}(B\cap f(B[x_0,\delta])),
$$ 
which is compact in $\R^d$. Since $M_\epsilon\to 1$ as $\epsilon\to 0^+$, the assertion of the lemma follows.
\end{proof}


\begin{proof}[Proof of Proposition \ref{Pr1}]
Let $W$ denote the $d$-dimensional $C^1$-manifold that contains $A$ and
let $\epsilon>0$ be arbitrary. In view of Definition \ref {D1}, for every point $x\in W$, there is an open neighborhood $V_x$ of $x$ relative to $W$ which is homeomorphic to an open set $U_x\subset \R^d$ such that the homeomorphism $f:U_x\to V_x$ is a $C^1$-continuous mapping and the Jacobian matrix $J^f_u$ (see the definition $J^{f}_u$ in (\ref{tgfr})) has rank $d$ for every $u\in U_x$. There is also a number $\epsilon_x>0$ such that $B(x,\epsilon_x)\cap W\subset V_x$. By Lemma \ref{lemma1}, there is a number $0<\delta (x)<\epsilon_x/2$ such that the set $B[x,2\delta (x)]\cap W=B[x,2\delta (x)]\cap f(U_x)$ is bi-Lipschitz homeomorphic to a compact set $D_x$ from $\R^d$ with constant $1+\epsilon$. Since $A$ is compact, the open cover $\{B(x,\delta (x))\}_{x\in A}$ has a finite subcover $\{B(x_i,\delta (x_i))\}_{i=1}^{p}$.

Denote $\delta_\epsilon:=\min\limits_{j= {1, \ldots,p}}\delta (x_j)$. Let $x$ be any point in $A$ and $r \in (0,\delta_\epsilon]$. There is an index $i$ such that $x\in B(x_i,\delta (x_i))$. Since $B(x,r)\cap A\subset B[x_i,2\delta (x_i)]\cap W$, the set $B(x,r)\cap A$ is bi-Lipschitz homeomorphic to a set $D_i\subset D_{x_i}$ with constant $1+\epsilon$. If $\varphi:B(x,r)\cap A\to D_i$ denotes the corresponding bi-Lipschitz mapping, we have $D_i\subset B(\varphi(x),(1+\epsilon)r)$. Then
$$
\mathcal H_d(B(x,r)\cap A)\leq (1+\epsilon)^d \mathcal L_d(D_i)\leq \beta_d r^d(1+\epsilon)^{2d}.
$$
Consequently,
$$
\OL \alpha_d (A;\delta_\epsilon)=\sup\limits_{r\in (0,\delta_\epsilon]}\sup\limits_{x\in A}\frac {\mathcal H_d(B(x,r)\cap A)}{\beta_d r^d}\leq (1+\epsilon)^{2d},
$$
which implies that $\lim\limits_{\delta\to 0^+}{\OL \alpha_d (A;\delta)}\leq 1$.
\end{proof}

\section{Acknowledgements}

 \quad The authors would like to thank Professor E.B. Saff who introduced this question to the authors and provided useful references and comments on this paper.


\end{document}